\documentclass[11pt]{article}
\usepackage{amsmath, amssymb, amsthm}
\usepackage{graphicx} 
\usepackage{enumitem}                                 % The preamble begins here.
\usepackage[margin=1.0in]{geometry}

\AtEndDocument{\bigskip{\footnotesize%
  \textsc{Einstein Institute of Mathematics, Edmond J. Safra Campus, The Hebrew University of Jerusalem,
Givat Ram. Jerusalem, 9190401, Israel.} \par  
  \textit{E-mail address} \texttt{ amir.algom@mail.huji.ac.il
} 
}}

\newtheorem{theorem}{Theorem}[section]

\newtheorem{Counter-example}[theorem]{Counter example}
\newtheorem{Claim}[theorem]{Claim}

\newtheorem{Lemma}[theorem]{Lemma}
\newtheorem{Proposition}[theorem]{Proposition}

\newtheorem*{theorem*}{Theorem}

\newcommand{\supp}{\text{supp}}
\newcommand{\diam}{\text{diam}}

\newcommand\blfootnote[1]{%
  \begingroup
  \renewcommand\thefootnote{}\footnote{#1}%
  \addtocounter{footnote}{-1}%
  \endgroup
}

\title{A simultaneous version of Host's equidistribution Theorem}

\author{Amir Algom}
\date{}
\begin{document}
\maketitle

\begin{abstract}
Let\blfootnote{Supported by ERC grant 306494 and ISF grant 1702/17.} $\mu$ be a  probability measure on $\mathbb{R}/\mathbb{Z}$ that is ergodic under the $\times p$ map,  with positive entropy. In 1995, Host \cite{Host1995normal} showed that if \blfootnote{2010 Mathematics Subject Classification 11K16, 11A63, 28A80, 28D05.} $\gcd(m,p)=1$  then $\mu$ almost every point is normal in base $m$. In 2001, Lindenstrauss \cite{Elon2001host} showed that the conclusion holds under the weaker assumption that $p$ does not divide any power of $m$. In 2015, Hochman and Shmerkin \cite{hochmanshmerkin2015} showed that this holds in the "correct" generality, i.e. if $m$ and $p$ are independent. We prove a simultaneous version of this result: for $\mu$ typical $x$, if $m>p$ are independent, we show that the orbit of $(x,x)$ under $(\times m, \times p)$ equidistributes for the product of the Lebesgue measure with $\mu$. We also show that if $m>n>1$ and $n$ is independent of $p$ as well, then the orbit of $(x,x)$ under $(\times m, \times n)$ equidistributes for the Lebesgue measure.
\end{abstract}

\section{Introduction}
\subsection{Background and main results}
Let $p$ be an integer greater or equal to $2$. Let $T_p$ be the $p$-fold map of the unit interval,  $$T_p (x) = p\cdot x \mod 1.$$
Let $m>1$ be an integer independent of $p$, that is, $\frac{\log p}{\log m} \notin \mathbb{Q}$. Henceforth, we will write $m \not \sim p$ to indicate that $m$ and $p$ are independent. In 1967 Furstenberg \cite{furstenberg1967disjointness} famously proved that if a closed subset of $\mathbb{T}:=\mathbb{R}/ \mathbb{Z}$ is jointly invariant under $T_p$ and $T_m$, then it is either finite or the entire space $\mathbb{T}$.  A well known Conjecture of Furstenberg about a measure theoretic analouge of this result, is that the only continuous probability measure jointly invariant under $T_p$ and $T_m$, and ergodic under the $\mathbb{Z}_+ ^2$ action generated by these maps, is the Lebesgue measure.  The best results towards this Conjecture, due to Rudolph \cite{Rudolph1990dis} for $p,m$ such that $\gcd(p,m)=1$ and later to Johnson \cite{Johnson1992dis} for $p \not \sim m$, is that it holds if in addition the measure has positive entropy with respect to the $\mathbb{Z}_+$ action generated by $T_p$ (see also the earlier results of Lyons \cite{Lyons1988dis}).

In 1995 Host proved the following pointwise strengthening of Rudolph's Theorem.  Recall that a number $x\in [0,1]$ is said to be normal in base $p$ if the sequence $\lbrace T_p ^k x \rbrace_{k \in \mathbb{Z}_+}$ equidistributes for the Lebesgue measure on $[0,1]$. Equivalently, the sequence of digits in the base $p$ expansion of $x$ has the same limiting statistics as an IID sequence of digits with uniform marginals.
\begin{theorem*} (Host, \cite{Host1995normal})
Let $p,m\geq 2$ be integers such that $\gcd(p,m)=1$. Let $\mu$ be $T_p$ invariant ergodic measure with positive entropy. Then $\mu$ almost every $x$ is normal in base $m$.
\end{theorem*}
Host's theorem can be shown to imply Rudolph's Theorem, but is more constructive in the sense that it proves that a large collection of measures satisfy a certain regularity property. Host's Theorem is also closely related to classical results of Cassels  \cite{Cassels1960normal} and Schmidt \cite{Schmidt1960normal} from around 1960, that proved a similar result for certain Cantor-Lebesgue type measures. This was later generalized by Feldman and Smorodinsky \cite{Feldman1992normal} to all non-degenerate Cantor-Lebesgue measures (in fact, weakly Bernoulli)  with respect to any base $p$ (though applying to a less general class of measures, the latter results nonetheless hold for any independent integers $p,m\geq 2$). We remark that the works of Meiri \cite{Meiri1998host} and of Hochman and Shmerkin \cite{hochmanshmerkin2015} contain excellent expositions on Host's Theorem, and on the  results of Cassels and Schmidt and some of the research that followed, respectively.

The assumption made on the integers $p,m$ in Host's Theorem, however, is stronger than it "should" be. Namely, it is stronger than assuming  that $p \not \sim m$. In 2001, Lindenstrauss \cite{Elon2001host} showed that the conclusion of Host's Theorem holds under the weaker assumption that $p$ does not divide any power of $m$. Finally, in 2015, Hochman and Shmerkin \cite{hochmanshmerkin2015} proved that Host's Theorem holds in the "correct" generality, i.e. when $p\not \sim m$.

Now, let $\mu$ be a measure as in Host's Theorem, with $p\not \sim m$. Then, on the one hand, by the results of Hochman and Shmerkin, for $\mu$ almost every $x$, its orbit under $T_m$ equidistributes for the Lebesgue measure. On the other hand, for $\mu$ almost every $x$, its orbit under $T_p$ equidistributes for $\mu$ (this is just the ergodic Theorem).  The main result of this paper is that this holds simultaneously. 
\begin{theorem} \label{Main Theorem}
Let $\mu$ be a $T_p$ invariant ergodic measure with $\dim \mu>0$.  Let $m>n>1$ be integers such that $m\not \sim p$. 
\begin{enumerate}
\item If $n=p$ then
\begin{equation*} 
\frac{1}{N} \sum_{i=0} ^{N-1} \delta_{(T_m ^i (x) , T_n ^i (x))} \rightarrow \lambda \times \mu, \quad \text{ for } \mu \text{ almost every } x,
\end{equation*} 
where the convergence is in the weak-* topology, and $\lambda$ is the Lebesgue measure on $[0,1]$.

\item If $n \not \sim p$ then
\begin{equation*} 
\frac{1}{N} \sum_{i=0} ^{N-1} \delta_{(T_m ^i (x) , T_n ^i (x))} \rightarrow \lambda\times \lambda, \quad \text{ for } \mu \text{ almost every } x.
\end{equation*} 
\end{enumerate}
\end{theorem}

Several remarks are in order. First,  the assumption that $\mu$ has positive dimension and the assumption that it has positive entropy are equivalent, so there is no discrepancy between the assumptions in Host's Theorem and those in Theorem \ref{Main Theorem} (see Section \ref{Section dimension} for a discussion on the dimension theory of measures).  Secondly, in the second part of the Theorem we do not need that $m$ and $n$ are independent, only that $m>n$. In addition, we can prove a version of Theorem \ref{Main Theorem}  where the initial point $(x,x)$ is replaced with $(f(x),g(x))$ for $f,g$ that are non singular affine maps of $\mathbb{R}$ that satisfy some extra mild conditions. This is explained in Section  \ref{Section pert}.

Theorem \ref{Main Theorem} can also be considered as part of the following general framework. Let $S,T \in \text{End}(\mathbb{T}^2)$, and let $\nu$ be an $S$ invariant probability measure. The idea is to study the orbits $\lbrace T^k x \rbrace_{k\in \mathbb{Z}_+}$ for $\mu$ typical $x$. In our situation,
\begin{equation*}
S = \begin{pmatrix}
p & 0 \\
0& p 
\end{pmatrix}, \quad 
T = \begin{pmatrix}
m & 0 \\
0& n 
\end{pmatrix}
\end{equation*}
and the measure $\nu=\Delta \mu$, where $\Delta:\mathbb{T}\rightarrow \mathbb{T}^2$ is the map $\Delta(x)=(x,x)$.

Problems around this framework were studied by several authors. Notable related examples are the works of Meiri and Peres \cite{Meiri1998Peres}, and the subsequnet work of Host \cite{Host2000high}. Meiri and Peres prove a Theorem similar to ours, with the following differences:
\begin{itemize}
\item They work with two general diagonal endomorphisms $S$ and $T$, but they require that the corresponding diagonal entries $S_{i,i}$ and $T_{i,i}$ be larger than $1$ and co-prime.

\item They allow for more general measures then the one dimensional measures that we work with (in Theorem \ref{Main Theorem} we work with measures on the diagonal of $\mathbb{T}^2$).
\end{itemize}
Host in turn has some requirements on $S$ and the measure that are more general than ours, but also requires that $\det (S)$ and $\det(T)$ be co-prime, and that for every $k$ the characteristic polynomial of  $T^k$ is irreducible over $\mathbb{Q}$ (clearly this is not the case here). The results of both Host, and Meiri and Peres, extend to any $d$ dimensional torus.

Our proof of Theorem \ref{Main Theorem} is inspired by the work of Hochman and Shmerkin \cite{hochmanshmerkin2015}. In particular,  the scenery of the measure $\mu$ at typical points plays a pivotal role in our work. We devote  the next Section to defining this scenery and some related notions, and to formulating the main technical tool used to prove Theorem \ref{Main Theorem}.

\subsection{On sceneries of measures and the proof of Theorem \ref{Main Theorem}} \label{Section scenrey flow}
We first recall some notions that were defined in (\cite{hochmanshmerkin2015}, Section 1.2). However, we remark that these notions and ideas have a long history, going back varisouly to Furstenberg (\cite{furstenberg1970intersections}, \cite{furstenberg2008ergodic}), Z\"{a}hle \cite{Zahle1988self}, Bedford and Fisher \cite{BedfordFisher1997Ratio}, M\"{o}rters and Preiss \cite{Morters1998Preiss}, and Gavish \cite{Gavish2008scaling}. See (\cite{hochmanshmerkin2015}, Section 1.2) and  \cite{hochman2010dynamics} for some further discussions and comparisons. 

For a compact metric space $X$ let $\mathcal{P}(X)$ denote the space of probability measures on $X$. Let
\begin{equation*}
\mathcal{M}^{\square} = \lbrace \mu \in \mathcal{P}([-1,1]):\quad 0\in \supp(\mu) \rbrace.
\end{equation*}
For $\mu \in \mathcal{M}^{\square}$ and $t\in \mathbb{R}$ we define the scaled measure $S_t \mu \in \mathcal{M}^{\square}$ by
\begin{equation*}
S_t \mu (E) = c \cdot \mu (e^{-t}E\cap [-1, 1]),\quad \text{where } c \text{ is a normalizing constant}.
\end{equation*}
For $x\in \supp (\mu)$ we similarly define the translated measure by
\begin{equation*}
\mu^x(E) = c' \cdot  \mu ( (E+x)\cap [-1,1]),\quad  \text{where } c' \text{ is a normalizing constant}.
\end{equation*}
The scaling flow is the Borel $\mathbb{R}^+$ flow $S=(S_t)_{t\geq0}$ acting on $\mathcal{M}^{\square} $. The scenery of $\mu$ at $x\in \supp(\mu)$ is the orbit of $\mu^x$ under $S$, that is, the one parameter family of measures $\mu_{x,t}:= S_t(\mu^x)$ for $t\geq0$. Thus, the scenery of the measure at some point $x$ is what one sees as one "zooms" into the measure.

Notice that $\mathcal{P}(\mathcal{M}^{\square} ) \subseteq \mathcal{P}(\mathcal{P}([-1,1]))$. As is standard in this context, we shall refer to elements of $\mathcal{P}(\mathcal{P}([-1,1]))$ as distributions, and to elements of $ \mathcal{P}(\mathbb{R})$ as measures. A measure $\mu \in P(\mathbb{R})$ generates a distribution $P\in \mathcal{P}(\mathcal{P}([-1,1]))$ at $x\in \supp(\mu)$ if the scenery at $x$ equidistributes for $P$ in $\mathcal{P}(\mathcal{P}([-1,1]))$, i.e. if
\begin{equation*}
\lim_{T\rightarrow \infty} \frac{1}{T} \int_0 ^T f(\mu_{x,t}) dt = \int f(\nu) dP(\nu),\quad \text{ for all } f\in C( \mathcal{P}([-1,1])).
\end{equation*}
and $\mu$ generates $P$ if it generates $P$ at $\mu$ almost every $x$.

If $\mu$ generates $P$, then $P$ is supported on $\mathcal{M}^{\square}$ and is $S$-invariant (\cite{hochman2010dynamics}, Theorem 1.7). We say that $P$ is trivial if it is the distribution supported on $\delta_0 \in \mathcal{M}^{\square}$ - a fixed point of $S$. To an $S$-invariant distribution $P$ we associate its pure point spectrum $\Sigma(P,S)$. This set consists of all the $\alpha \in \mathbb{R}$ for which there exists  a non-zero measurable function $\phi:\mathcal{M}^{\square} \rightarrow \mathbb{C}$ such that $\phi\circ S_t = \exp(2 \pi i \alpha t)\phi$ for every $t\geq 0$, on a set of full $P$ measure. The existence of such an eigenfunction indicates that some non-trivial feature of the measures of $P$ repeats periodically under magnification by $e^\alpha$.

Finally, we say that a measure $\mu \in \mathcal{P}([0,1])$  is pointwise generic under $T_n$ for a  measure $\rho \in \mathcal{P}([0,1])$ if  $\mu$ almost every $x$ equidistributes for $\rho$ under $T_n$, that is,
\begin{equation*}
\frac{1}{N} \sum_{k=0} ^{N-1} f(T^k _n x) \rightarrow \int f(x) d\rho(x),\quad \forall f\in C([0,1]).
\end{equation*}
We are now  ready to state our second main result, which is the technical tool that shall be employed to prove Theorem \ref{Main Theorem}.
\begin{theorem} \label{Conjecture}
Let $\mu \in \mathcal{P}([0,1])$ and let $m>n>1$ be integers, such that:
\begin{enumerate}
\item The measure $\mu$ generates a non-trivial $S$-ergodic distribution $P\in \mathcal{P}(\mathcal{P}([-1,1]))$.

\item The pure point spectrum $\Sigma(P,S)$ does not contain a non-zero integer multiple of $\frac{1}{\log m}$.

\item  The measure $\mu$ is pointwise generic under $T_n$ for an ergodic and continuous measure  $\rho$.
\end{enumerate}
Then 
\begin{equation} \label{Eq I1}
\frac{1}{N} \sum_{i=0} ^{N-1} \delta_{(T_m ^i (x) , T_n ^i (x))} \rightarrow \lambda\times \rho, \quad \text{ for } \mu \text{ almost every } x.
\end{equation}
\end{theorem}

Notice that under assumption (3) of Theorem \ref{Conjecture}, the measure $\rho$ is $T_n$ invariant, so its ergodicity is with respect to this map. Theorem \ref{Conjecture} together with the machinery developted by Hochman and Shmerkin in (\cite{hochmanshmerkin2015},  Section 8) imply Theorem \ref{Main Theorem}: this is explained in Section \ref{Section final proof}. 

We end this introduction with a brief overview of the proof of Theorem \ref{Conjecture}. First, we note that if we only assume (1) and (2) in Theorem \ref{Conjecture}, then 
\begin{equation}  \label{Eq HS}
\frac{1}{N} \sum_{i=0} ^{N-1} \delta_{T_m ^i (x)} \rightarrow \lambda, \quad \text{ for } \mu \text{ almost every } x,
\end{equation}
according to the main result of Hochman and Shmerkin \cite{hochmanshmerkin2015}. This is proved by roughly following three steps: first, using the spectral condition,  they show that any accumulation point $\nu$ of measures as in \eqref{Eq HS} can be represented as an integral over measures that are closely related to those drawn according to $P$. They proceed to use this representation to show that
\begin{equation*}
\text{There exists some } \epsilon >0 \text{ such that for any } \tau\in \mathcal{P}([0,1]) \text{ with } \dim \tau \geq \epsilon,  \dim \tau*\nu=1.
\end{equation*}
They conclude by showing that the only $T_m$ invariant  measure $\nu$ satisfying the latter property is the Lebesgue measure. 

Our strategy is to first show that  a $T_m \times T_n$ invariant measure $\nu$, that projects to $\lambda$ and to $\rho$ in the first and second coordinate respectively, must be $\lambda\times \rho$ if it satisfies the following condition:
\begin{equation*}
\text{There exists some } \epsilon>0 \text{ such that for any } \tau\in \mathcal{P}([0,1]) \text{ with } \dim \tau \geq \epsilon,  \dim \tau*P_1\nu_y=1,
\end{equation*}
for $\rho$ almost every $y$, where $\nu_y$ is the conditional measure of $\nu$ on the fiber $\lbrace (x,z): z=y\rbrace$, and $P_1(x,y)=x$.  This is Claim \ref{Claim rigidty} in Section  \ref{Section rig}. Afterwards, we show that this property holds for all the accumulation points of the measures from \eqref{Eq I1}. This is done via a corresponding integral representation, see Claim \ref{Theorem 5.1} in Section \ref{Section rep}.

\textbf{Notation} We shall use the letter $\lambda$ to indicate both the Lebesgue measure on $[0,1]$ and the Lebesgue measure on $\mathbb{T}$.  Which is meant will be clear from context. Also, whenever we have a finite product space, we denote by $P_i$ the projection to the $i$-th coordinate.

\textbf{Acknowledgements} I would like to thank Mike Hochman for suggesting the problem studied in this paper to me, and for his various insightful comments on earlier versions of this manuscript.  I would also like to thank Shai Evra for some helpful conversations related to this paper. 

\section{Preliminaries}
 \subsection{Dimension theory of measures, and their Fourier coefficients} \label{Section dimension}
For a Borel set $A$ in some metric space $X$, we denote by $\dim A$ its Hausdorff dimension, and by $\dim_P A$ its packing dimension (see Falconer's book \cite{falconer1986geometry} for an exposition on these concepts). Now,  let $\mu \in \mathcal{P}(X)$. The (lower) Hausdorff dimension of the measure $\mu$ is defined as
\begin{equation*} 
\dim \mu := \inf \lbrace \dim A:\quad \mu(A)>0,\quad A \text{ is Borel} \rbrace,
\end{equation*}
and the upper Hausdorff dimension of the measure $\mu$ is defined as
\begin{equation*} 
\overline{\dim} \mu := \inf \lbrace \dim A:\quad \mu(A)=1,\quad A \text{ is Borel} \rbrace,
\end{equation*}
The (upper) packing dimension of the measure $\mu$ is defined as
\begin{equation*}
\dim_P \mu = \inf \lbrace \dim_P A : \mu(A)=1,\quad A \text{ is Borel} \rbrace.
\end{equation*}

An alternative characterization of the dimension of $\mu$ that we shall often use is given in terms of their local dimensions:  For every $x\in \supp(\mu)$  we define the   local (pointwise) dimension of $\mu$ at $x$  as
\begin{equation*}
\dim(\mu,x)=\liminf_{r\rightarrow 0} \frac{\log \mu (B(x,r))}{\log r}
\end{equation*}
where $B(x,r)$ denotes the closed ball or radius $r$ about $x$. The Hausdorff dimension of $\mu$ is equal to
\begin{equation} \label{Eq lower dim}
\dim \mu = \text{ess-inf}_{x\sim \mu} \dim(\mu,x),
\end{equation} 
and  the upper Hausdorff dimension of $\mu$ is equal to
\begin{equation} \label{Eq upper dim}
\overline{\dim} \mu = \text{ess-sup}_{x\sim \mu} \dim(\mu,x).
\end{equation} 
see e.g. \cite{falconer1997techniques}. If $\dim (\mu,x)$ exists as a limit at almost every point, and is constant almost surely, we shall say that the measure $\mu$ is exact dimensional. In this case, most metric definitions of the dimension of $\mu$ coincide (e.g. lower and upper Hausdorff dimension and packing dimension).

Let us now collect some known facts regarding dimension theory of measures:
\begin{Proposition} \label{Prop properties of dim}
\begin{enumerate}
\item Let $\mu\in \mathcal{P}(\mathbb{R}^d)$ and suppose that there is a  distribution $Q\in \mathcal{P}(\mathcal{P}(\mathbb{R}^d))$ such that
\begin{equation*}
\mu = \int \nu dQ(\nu).
\end{equation*}
Then 
\begin{equation*}
\dim \mu \geq \text{ess-inf}_{\nu \sim Q} \dim \nu
\end{equation*}
If $\mu=p_1\mu_1+p_2 \mu_2$ where $\mu_{1},\mu_{2}\in \mathcal{P}(\mathbb{R}^d)$ and  $(p_1,p_2)$ is any probability vector, then
\begin{equation*}
\dim \mu = \min \lbrace \dim \mu_1, \dim \mu_2 \rbrace.
\end{equation*}

\item Let $f:X\rightarrow Y$ be a Lipschitz map between complete metric spaces. Then for any $\mu\in \mathcal{P}(X)$,
\begin{equation*}
\dim f\mu \leq \dim \mu
\end{equation*}
with an equality if $f$ is locally bi-Lipschitz.

\item Let $\mu \in \mathcal{P}(\mathbb{T})$ be exact dimensional, and $\nu \in \mathcal{P}(\mathbb{T})$ be a measure supported on finitely many atoms. Then $\dim \mu*\nu = \dim \mu$, and moreover, $\mu*\nu$ is exact dimensional.
\end{enumerate}
\end{Proposition}

The next Lemma is essentially Lemma 3.5 in \cite{hochmanshmerkin2015}, with a minor modification  which follows e.g. from Lemma 6.13 in \cite{hochman2010dynamics}.
\begin{Lemma} (\cite{hochmanshmerkin2015}, Lemma 3.5) \label{Lemma 3.5}
Let $\mu\in \mathcal{P}(\mathbb{R}^2)$.
\begin{enumerate}
\item Suppose that for $P_2 \mu$ almost every $y$, $\dim \mu_y \geq \alpha$ (where $\mu_y$ is the conditional measure on the fiber $\lbrace (x,z):z=y\rbrace$). Then $\dim \mu \geq \dim P_2 \mu + \alpha$.  

\item For an upper bound, we have $\dim \mu \leq \dim_P P_1 \mu + \dim P_2 \mu$.
\end{enumerate}
\end{Lemma}

We end this section with a brief discussion of the Fourier coefficients of measures on $\mathbb{T}^d$. These are defined as follows. First, given $\mu \in \mathcal{P}(\mathbb{T}^d)$ we define for any $k \in \mathbb{Z}^d$ the corresponding Fourier coefficient by
\begin{equation*}
\hat{\mu}(k):=\int e^{2 \pi i k\cdot x} d \mu (x).
\end{equation*}
The following relations are easily verified for two measures $\mu,\nu\in \mathcal{P}(\mathbb{T})$:
\begin{equation} \label{Eq fourier conv}
\widehat{\mu*\nu}(k) = \hat{\mu}(k)\cdot \hat{\nu}(k), \quad  k\in \mathbb{Z}.
\end{equation}
\begin{equation} \label{Eq fouier prod}
\widehat{\mu \times \nu}(k,j) =  \hat{\mu}(k)\cdot \hat{\nu}(j), \quad  (k,j)\in \mathbb{Z}^2.
\end{equation}

The following Lemma is standard:
\begin{Lemma} (\cite{Mattila2015new}, Section 3.10 ) \label{Lemma fou}
Let $\mu,\nu\in \mathcal{P}(\mathbb{T}^2)$. If $\hat{\mu}(k)=\hat{\nu}(k)$ for all $k\in \mathbb{Z}^2$ then $\mu=\nu$.
\end{Lemma}

Finally, let $m\geq 2$ and let $\mu$ be the Cantor-Lebesgue measure corresponding to the non-degenerate probability vector $(p_0,...,p_{m-1})$. That is, $\mu$ is the distribution of the Random sum $\sum_{k=1} ^\infty \frac{X_k}{m^k}$, where $X_k$ are IID random variables with $P(X_k = i) = p_i$. It is a well known fact that for every $k\in \mathbb{Z}$,
\begin{equation} \label{Eq fouier cantor leb}
\hat{\mu} (k) = \prod_{j=1} ^\infty \left( \sum_{u=0} ^{m-1} p_u \exp(2 \pi i u \frac{k}{m^j} ) \right).
\end{equation}
 
 \subsection{Dimension theory of invariant measures} \label{Section dim of erg}
\subsubsection{Some notions from ergodic theory} 
In this paper, a dynamical system is a quadruple $(X, \mathcal{B},T,\mu)$, where $X$ is a compact metric space, $\mathcal{B}$ is the Borel sigma algebra, and $T:X\rightarrow X$ is a  measure preserving map, i.e. $T$ is Borel measurable and $T\mu = \mu$. Since we always work with the Borel sigma-algebra, we shall usually just write $(X,T,\mu)$. For example one may consider $X=\mathbb{T}$, the Borel map $T_p$ for some $p\geq 2$, and  some Cantor-Lebesgue measure with respect to base $p$.

A dynamical system is ergodic if and only if the only invariant sets are trivial. That is, if $B\in \mathcal{B}$ satisfies $T^{-1} (B) = B$ then $\mu(B)=0$ or $\mu(B)=1$. A dynamical system is called weakly mixing if for any ergodic dynamical system $(Y,S,\nu)$, the product system $(X\times Y, T\times S, \mu\times \nu)$ is also ergodic. In particular, weakly mixing systems are ergodic. Moreover, If both $(X,T,\mu)$ and $(Y,S,\nu)$ are weakly mixing, then their product system is also weakly mixing. A class of examples of  weakly mixing systems is given $(\mathbb{T},T_p,\mu)$ where $\mu$ is a Cantor-Lebesgue measure with respect to base $p$.

We will have occasion to use  the ergodic decomposition Theorem: Let $(X,T,\mu)$ be a dynamical system. Then there is a  map $X\rightarrow \mathcal{P}(X)$, denoted by $\mu \mapsto \mu^{x}$, such that:
\begin{enumerate}
\item The map $x\mapsto \mu^x$ is measurable with respect to the sub-sigma algebra $\mathcal{I}$ of $T$ invariant sets.

\item $\mu = \int \mu^x d\mu(x)$

\item For $\mu$ almost every $x$, $\mu^x$ is $T$ invariant and ergodic. The measure $\mu^x$ is called the ergodic component of $x$.
\end{enumerate}

Finally, we shall say that a point $x\in X$ is generic with respect to $\mu$ if
\begin{equation*}
\frac{1}{N} \sum_{i=0} ^{N-1} \delta_{T^i x} \rightarrow \mu,\quad \text{ where } \delta_y \text{ is the dirac measure on } y\in X,
\end{equation*}
in the weak-* topology. By the ergodic Theorem, if $\mu$ is ergodic then $\mu$ a.e. $x$ is generic for $\mu$.

\subsubsection{Dimension theory of invariant measures}
Recall that in general, if  $\mu\in \mathcal{P}(X)$ is a $T$ invariant measure, we may define its entropy with respect to $T$, a quantity that we shall denote by $h(\mu,T)$. As there is an abundance of excellent texts on entropy theory (e.g. \cite{Walters1982ergodic}), we omit a discussion on entropy  here. We now restrict our attention to dynamical systems of the form $(\mathbb{T}, \mu, T_p)$ or $(\mathbb{T}^2 ,\mu, T_m \times T_n)$, where we always assume that $m>n>1$. In the one dimensional case, the dimension of $\mu$ may be computed via the entropies of its ergodic components:
\begin{theorem} (\cite{Elon1999conv}, Theorem 9.1) \label{Theorem 9.1}
Let $\mu\in \mathcal{P}(\mathbb{T})$ be a $T_p$ invariant and ergodic measure.  Then $\mu$ is exact dimensional and
\begin{equation*}
\dim \mu = \frac{h(\mu,T_p)}{\log p}.
\end{equation*}
In general, if $\mu\in \mathcal{P}(\mathbb{T})$ is a $T_p$ invariant measure with ergodic decomposition $\mu = \int \mu^x d\mu(x)$, then
\begin{equation} \label{Eq dim of inv measu}
\dim \mu = \text{ess-inf}_{x\sim \mu} \dim \mu^{x}
\end{equation}
and 
\begin{equation} \label{Eq upper dim of inv measu}
\overline{\dim} \mu = \text{ess-sup}_{x\sim \mu} \dim \mu^{x}
\end{equation}
\end{theorem}  

The situation for dynamical systems of the form $(\mathbb{T}^2, \mu, T_m \times T_n)$ is more complicated. This may be attributed to the fact that the map $T_m \times T_n$ is not conformal. There is, however, a way to compute the dimension of $\mu$ in this situation via entropy theory, using a suitable version of the Ledrappier-Young formula . This was first done  by Kenyon and Peres in \cite{Peres1996measures} for ergodic measures. The general case may be treated using similar methods,  as observed by Meiri and Peres in (\cite{Meiri1998Peres}, Lemma 3.1). 
\begin{theorem} \cite{Meiri1998Peres} \label{Theorem Led young}
Let $\mu\in \mathcal{P}(\mathbb{T}^2)$ be a $T_m \times T_n$ invariant measure. Then for $\mu$ almost every $x$ the local dimension $\dim (\mu,x)$ exists as a limit and
\begin{equation*}
\dim(\mu,x)= \frac{h(\mu^x, T_m\times T_n)-h(( P_2 \mu ) ^{P_2  x}, T_n)}{\log m}+\frac{ h(( P_2 \mu ) ^{P_2  x}, T_n)}{\log n}
\end{equation*}
where $\mu^x$ and $( P_2 \mu ) ^{P_2  x}$ denote the corresponding ergodic components of $\mu$, and of $P_2 \mu$, respectively.
\end{theorem} 

Finally, we will require the following result of Meiri, Lindenstrauss and Peres from \cite{Elon1999conv}:
\begin{theorem} \cite{Elon1999conv}  \label{Theorem Elon conv}
Let $\mu\in \mathcal{P}(\mathbb{T})$ be a $T_p$ invariant weakly mixing measure, such that $\dim \mu >0$. Let $\mu^{*k}$ denote the convolution of $\mu$ with itself $k$-times. Then
\begin{equation*}
\dim( \mu^{*k}) \rightarrow 1,\quad \text{ monotonically  as } k\rightarrow \infty 
\end{equation*} 
\end{theorem} 
We remark that we have only cited a special case of this result. Indeed, Meiri, Lindenstrauss and Peres deal with the growth of the entropy of more general convolutions of $T_p$ ergodic measures. We refer the reader to \cite{Elon1999conv} for the full statement.

\subsection{Relating the distribution of orbits to the measure} \label{Section relating}
Let $X$ be a compact metric space, $T:X\rightarrow X$ a Borel measurable map, and let $\mu,\nu\in \mathcal{P}(X)$. Following Hochman and Shmerkin \cite{hochmanshmerkin2015}, we shall say that $\mu$ is pointwise generic for $\nu$ under $T$ if $\mu$ almost every $x$ equidistributes for $\nu$ under $T$, that is,
\begin{equation*}
\frac{1}{N} \sum_{k=0} ^{N-1} f(T^k x) \rightarrow \int f d\nu,\quad \forall f\in C(X).
\end{equation*}
This notion is closely related to the main results of this paper. Indeed, let $X=\mathbb{T}^2$,  $T=T_m \times T_p$ for $m>p>1$ and $m\not \sim p$, and  $\alpha$ be the pushforward of a $T_p$ invariant ergodic positive dimensional measure $\mu \in \mathcal{P}(\mathbb{T})$ to the diagonal of $\mathbb{T}^2$. Then Theorem \ref{Main Theorem} part (1) for example may be stated as "$\alpha$ is pointwise generic for $\lambda \times \mu$ under $T$".

In \cite{hochmanshmerkin2015}, the authors obtain a criteria for this to occur, one that shall play a central role in this paper as well. We now recall its formulation.  Let $\mathcal{A}$ be a finite partition of $X$, and for every $i\in \mathbb{N}\cup \lbrace 0 \rbrace$ let $T^i \mathcal{A}= \lbrace T^{-i} A: A\in \mathcal{A}\rbrace$. Let $\mathcal{A}^k = \bigvee_{i=0} ^{k-1} T^i \mathcal{A}$ denote the coarsest common refinement of $\mathcal{A}, T^{1} \mathcal{A}...,T^{k-1}\mathcal{A}$.  Now, if the smallest sigma algebra that contains $\mathcal{A}^k$ for all $k$ is the Borel sigma algebra, we say that $\mathcal{A}$ is a generator for $T$. We say that $\mathcal{A}$ is a topological generator if $\sup \lbrace \diam A: A\in \mathcal{A}^k \rbrace \rightarrow 0$ as $k\rightarrow \infty$. A topological generator is clearly a generator.

Let us give two examples of topological generators that shall be used in this paper: for every $p\in \mathbb{N}$ let $\mathcal{D}_{p}$ be the $p$-adic partition of $\mathbb{T}$ (and of $\mathbb{R}$), that is,
\begin{equation*}
\mathcal{D}_p = \lbrace [\frac{z}{p}, \frac{z+1}{p}) :\quad  z\in \mathbb{Z} \rbrace.
\end{equation*}
Then, under the map $T_p$, we see that
\begin{equation*}
\mathcal{D}_p ^k = \mathcal{D}_{p^k} = \lbrace [\frac{z}{p^k}, \frac{z+1}{p^k}) :\quad  z\in \mathbb{Z} \rbrace.
\end{equation*}
It is thus easy to see that $\mathcal{D}_p$ is a generator for $T_p$. Similarly, if $m>n$ then the partition $\mathcal{D}_{m} \times \mathcal{D}_{n}$ of $\mathbb{T}^2$ is a generator under $T_m \times T_n$.

  Finally, in general, for every $k\geq 1$ and $x\in X$, let $\mathcal{A}^k (x)$ denote the unique element of $\mathcal{A}^k$ that contains $x$. Given $\mu \in \mathcal{P}(X)$ and $x\in X$ such that $\mu(\mathcal{A}^k (x))>0$, let
  \begin{equation*}
  \mu_{\mathcal{A}^k (x)} = c\cdot T^k( \mu|_{\mathcal{A}^k (x)}),\quad \text{ where } c=\mu (\mathcal{A}^k (x))^{-1},
  \end{equation*}
which is well defined almost surely.
  
\begin{theorem} (\cite{hochmanshmerkin2015}, Theorem 2.1) \label{Theorem 2.1}
Let $T:X\rightarrow X$ be a Borel measurable map of a compact metric space, $\mu \in \mathcal{P}(X)$ and $\mathcal{A}$ a generating partition. Then for $\mu$ almost every $x$, if $x$ equidistributes for $\nu\in \mathcal{P}(X)$ along some $N_k\rightarrow \infty$, and if $\nu(\partial A)=0$ for all $A\in \mathcal{A}^k,k\in \mathbb{N}$, then
\begin{equation*}
\nu = \lim_{k\rightarrow \infty} \frac{1}{N_k} \sum_{k=0} ^{N_k-1} \mu_{\mathcal{A}^k (x)},\quad \text{weak-* in } \mathcal{P}(X).
\end{equation*}
\end{theorem}

A crucial ingredient in our application of Theorem \ref{Theorem 2.1} is the following Claim. Let $m>n$, and define for every $k\in \mathbb{N}$
\begin{equation*}
A_k = \lbrace x \in \mathbb{R} :\quad \mathcal{D}_{m^k} (x) \not \subseteq \mathcal{D}_{n^k} (x) \rbrace
\end{equation*}
Also, recall that the density of a sequence $S\subseteq \mathbb{N}$ (if it exists) is the limit of the sequence $\frac{|S\cap [1,N]|}{N}$ as $N\rightarrow \infty$. If the limit does not exist, the corresponding $\limsup$ is called the upper density of $S$.
\begin{Claim} \label{Claim stable intersections}
Suppose that $\mu\in \mathcal{P}([0,1])$ is a measure that is pointwise generic under $T_n$ for a continuous measure $\rho$. Then for $\mu$ almost every $x$, if $x\in \limsup A_k$ and $\lbrace n_k \rbrace$ represents the times when $x\in A_{n_k}$, then the density of $\lbrace n_k \rbrace$ is zero.
\end{Claim} 
\begin{proof}
Choose $x\sim \mu$, and if $x\in \limsup A_k$ let $\lbrace n_k \rbrace$ be the sequence as in the statement of the Claim. Let $\epsilon>0$. We will show that the upper density of $\lbrace n_k \rbrace$ is at most $\epsilon$. First, since $\rho$ is a continuous measure, there exists some $\delta>0$ such that $\rho (B(0,\delta) ) <\epsilon$, where $B(0,\delta)$ is the ball about $0$ in $\mathbb{T}$. By our assumption that $\mu$ is pointwise generic under $T_n$ for $\rho$, and since $\rho$ is a continuous measure,
\begin{equation*}
V_\delta = \lbrace i| \quad T_n ^i (x) \in B(0,\delta) \rbrace
\end{equation*}
has density $\rho(B(0,\delta) ) <\epsilon$.

Now, let us decompose our sequence
\begin{equation*}
\lbrace n_k \rbrace = \left( \lbrace n_k \rbrace \cap V_\delta \right) \cup \left( \lbrace n_k \rbrace \cap \left( \mathbb{N} \setminus V_\delta \right) \right). 
\end{equation*}
Then the upper density of $\lbrace n_k \rbrace \cap V_\delta$ is at most $\epsilon$. We now show that the density of the sequence $\lbrace \ell_k \rbrace :=\lbrace n_k \rbrace \cap \left( \mathbb{N} \setminus V_\delta \right)$ is $0$. In fact, we will show that this is a finite sequence.

Indeed, let $K > \frac{\log \delta}{\log \frac{n}{m}}$. We claim that $\lbrace \ell_k \rbrace \subseteq [0, K]$. Assume towards a contradiction that there exists some $q>K$ such that $\ell_k =q$ for some $k$. Then there is a unique $n^q$-adic number $a$ (an endpoint of an $\mathcal{D}_{n^q}$ cell) such that $a\in \mathcal{D}_{m^q} (x)$. Write $a = \frac{s}{n^{q}}$ for some integer $s$. Then we have
\begin{equation*}
|x-\frac{s}{n^{q}}| \leq \frac{1}{m^{q}},
\end{equation*}
which implies that $T_n ^q (x) \in B(0, \frac{n^q}{m^q}) \subset B(0,\delta)$, by the choice of $K$. Thus, $q\in V_\delta$, contradicting the choice of the sequence $\lbrace \ell_k \rbrace$. Thus, $\lbrace \ell_k \rbrace \subseteq [0, K]$, which is sufficient for us.
\end{proof}

We will also require the following Lemma.
\begin{Lemma} \label{Lemma bad indices}
Let $x\in [0,1]$ be such that it equidistributes for a continuous measure $\rho$ under $T_n$. Let $D\subset [0,1]$ be some interval. Let $\lbrace n_k \rbrace$ be the sequence of times when $T_n ^{n_k} x\notin \overline{D}$ but $d(T_n ^{n_k} (x), \partial D) \leq (\frac{n}{m})^{n_k}$. Then the density of the sequence $\lbrace n_k \rbrace$ is $0$.
\end{Lemma}
\begin{proof}
Let $\epsilon>0$. Since $\rho$ is continuous, there exists some $\delta>0$ such that 
\begin{equation*}
\rho (\lbrace y:\quad d(y, \partial D) \leq \delta \rbrace ) < \epsilon.
\end{equation*}
Let
\begin{equation*}
V_\delta = \lbrace k|\quad T_n ^k x \in \lbrace y:\quad d(y, \partial D) \leq \delta \rbrace \quad  \rbrace.
\end{equation*}
Then by our assumption on $x$, the density of $V_\delta$ is at most $\epsilon$. However, the sequence $\lbrace n_k \rbrace \subseteq V_\delta$, apart from maybe finitely many indices. It follows that the upper density of $\lbrace n_k \rbrace$ is at most the density of $V_\delta$, and therefore is at most $\epsilon$. This proves the Lemma. 
\end{proof}
\subsection{Ergodic fractal distributions} \label{Section EFD}
Recall the definitions introduced in Section \ref{Section scenrey flow}. In this Section we discuss some other related results of \cite{hochmanshmerkin2015} that we shall require.  First, we cite a result about the implication of not having some element $t_0>0$ in the pure point spectrum of a distribution generated by a measure. 
\begin{Proposition} \label{Prop spec assumption} (\cite{hochmanshmerkin2015}, Section 4) Suppose that $\mu$ generates an $S$-ergodic distribution $P$ and that no non-zero integer multiple of $t_0>0$ is in $\Sigma(P,S)$. Then $P$ is $t_0$-generated by $\mu$ at almost every $x$, i.e. the sequence $\lbrace \mu_{x,kt_0} \rbrace_{k=0} ^\infty$ equidistributes for $P$. 
\end{Proposition}

The next result says that distributions $P\in \mathcal{P}(\mathcal{P}([0,1]))$ that are generated by a given measure $\mu$ have some additional invariance properties:
\begin{theorem}  \label{Theorem 4.7} (\cite{hochmanshmerkin2015}, Theorem 4.7) Suppose that $\mu$ generates an $S$-invariant distribution $P$. Then $P$ is supported on $\mathcal{M}^\square$ and satisfies the $S$-quasi-Palm property: for every Borel set $B\subseteq \mathcal{M}^\square$, $P(B)=1$ if and only if for every $t>0$, $P$ almost every measure $\eta$ satisfies that $\eta_{x,t}\in B$ for $\eta$ almost every $x$ such that $[x-e^{-t},x+e^{-t}]\subseteq [-1,1]$.  
\end{theorem} 

We shall refer henceforth to $S$-ergodic distributions $P$ supported on $\mathcal{M}^\square$ that satisfy the conclusion of Theorem \ref{Theorem 4.7} as EFD's (Ergodic Fractal Distributions), a term coined by Hochman in \cite{hochman2010dynamics}. The next Proposition says that typical measures with respect to a non-trivial EFD have positive dimension (recall the definition of non-triviality in this situation from Section \ref{Section scenrey flow}):
\begin{Proposition} \label{Prop posit dim for EFD} \label{Prop positive dim for EFD} (\cite{hochmanshmerkin2015}, Proposition 4.12)
Let $P$ be an EFD. Then there exists some $\delta\geq 0$ such that $P$ almost every $\nu$ has $\dim \nu = \delta$. If $P$ is non-trivial then $\delta>0$.
\end{Proposition}

We will also need to know that $P$-typical measures are not "one sided at small scales"
\begin{Proposition} (\cite{hochmanshmerkin2015}, Proposition 4.13) \label{Prop tight}
Let $P$ be an EFD. For every $\rho>0$, for $P$ almost every $\nu$, we have $\inf \nu(I)>0$, where $I\subseteq [-1,1]$ ranges over closed intervals of length $\rho$ containing $0$.
\end{Proposition}

The next Proposition follows from the $S$-invariance of EFD's, and from a Theorem of Hunt and Kaloshin \cite{Hunt1997Kaloshin}:
\begin{Proposition} \label{Lemma 5.8} (\cite{hochmanshmerkin2015}, Lemma 5.8)
Let $P$ be a non trivial EFD such that $P$ typical measures have dimension $\delta>0$. Let $\tau\in \mathcal{P}(\mathbb{R})$ be such that $\dim \tau \geq 1-\delta$. Then $\dim \tau * \nu =1$ for $P$ almost every $\nu$.
\end{Proposition}

Finally, the next Proposition shows that ergodic $T_p$ invariant measures of positive dimension generate non-trivial EFD's:
\begin{theorem} \cite{hochman2010geometric} \label{Theorem 2.4}
Let $\mu\in \mathcal{P}([0,1])$ be a $T_p$ invariant ergodic measure with $\dim \mu>0$. Then $\mu$ generates a non-trivial $S$ ergodic distribution $P$ (which is an EFD by Theorem \ref{Theorem 4.7}). 
\end{theorem}
Let $m\not \sim p$. We remark that while non-degenerate Cantor-Lebesgue measures with respect to base $p$ do generate EFD's $P$ such that $\frac{k}{\log m} \notin \Sigma(P,S)$ for every non zero integer $k$, this is not true in general. Thus, in order to deduce Theorem \ref{Main Theorem} from Theorem \ref{Conjecture}, we shall require some additional machinery developed by Hochman and Shmerkin in \cite{hochmanshmerkin2015} for a similar purpose. This is discussed in Section \ref{Section final proof}.

\section{Some properties of (times m, times n) invariant measures} \label{Section rig}
Throughout this section we fix integers $m>n>1$. We begin with an elementary Lemma from entropy theory. Recall that we denote the coordinate projections by $P_1,P_2$.
\begin{Lemma} \label{Lemma entropy}
Let $\alpha\in \mathcal{P}(\mathbb{T}^2)$ be a $T_m \times T_n$ invariant  measure such that  $P_2 \alpha = \rho$. If  
\begin{equation*}
h(T_m \times T_n, \alpha) = \log m + h(T_n,\rho)
\end{equation*}
then $\alpha = \lambda \times \rho$.
\end{Lemma}
\begin{proof}
Let  $\mathcal{E}$ be the invariant sigma algebra that corresponds to the second coordinate of $\mathbb{T}^2$. Then, by the Abramov-Rokhlin Lemma (see \cite{Crauel1992formula} for the non-invertible case),
\begin{equation*}
h(T_m \times T_n, \alpha) = h(T_m \times T_n, \alpha |\mathcal{E})+h(T_n,\rho).
\end{equation*}
Combining this with our condition, we see that 
\begin{equation*}
h(T_m \times T_n, \alpha |\mathcal{E}) = \log m.
\end{equation*}
Recall that the partition $\mathcal{A} = \mathcal{D}_m \times \mathcal{D}_n$ is a generating partition of $\mathbb{T}^2$ (see Section \ref{Section relating}). Then it follows from Fekete's Lemma and the Kolmogorov-Sinai Theorem that
\begin{equation*}
\inf_k \frac{1}{k} H_\alpha ( \bigvee_{i=0} ^{k-1} (T_m\times T_n)^{i} \mathcal{A} | \mathcal{E})= h(T_m \times T_n, \mathcal{A}, \alpha |\mathcal{E}) = h(T_m \times T_n, \alpha |\mathcal{E}) = \log m.
\end{equation*}
As $\log m$ is also an upper bound for the sequence $\lbrace \frac{1}{k} H_\alpha ( \bigvee_{i=0} ^{k-1} (T_m\times T_n)^{i} \mathcal{A} | \mathcal{E}) \rbrace$, we find that for every $k\in \mathbb{N}$,
\begin{equation*}
\frac{1}{k} H_\alpha ( \bigvee_{i=0} ^{k-1} (T_m\times T_n)^{i} \mathcal{A} | \mathcal{E}) = \log m.
\end{equation*}
So, by the formula for conditional entropy as average of the conditional measures $\lbrace \alpha_x ^{\mathcal{E}} \rbrace$,
\begin{equation*}
\log m ^k = H_\alpha ( \bigvee_{i=0} ^{k-1} (T_m\times T_n)^{i} \mathcal{A}| \mathcal{E} ) = \int H_{\alpha_x ^{\mathcal{E}}} ( \bigvee_{i=0} ^{k-1} (T_m\times T_n)^{i} \mathcal{A} ) d\rho (x)= \int H_{\alpha_x ^{\mathcal{E}}} ( \bigvee_{i=0} ^{k-1} T_m^{i} \mathcal{D}_m  ) d\rho (x),
\end{equation*}
where the partition in the last term on the RHS should be understood as the corresponding partition on the fiber $[0,1]\times \lbrace  P_2 (x) \rbrace$. We also have $H_{\alpha_x ^{\mathcal{E}}} ( \bigvee_{i=0} ^{k-1} T_m^{i} \mathcal{D}_m ) \leq \log m^k$ almost surely, since $\bigvee_{i=0} ^{k-1} T_m^{i} \mathcal{D}_m$ has $m^k$ atoms. Therefore,
\begin{equation*}
H_{\alpha_x ^{\mathcal{E}}} ( \bigvee_{i=0} ^{k-1} T_m^{i} \mathcal{D}_m  ) = \log m ^k 
\end{equation*}
almost surely.  Such an equality is possible only if  $\alpha_x ^{\mathcal{E}}$ is the uniform measure on $\mathcal{D}_{m^k}$. It follows that almost surely the measure $\alpha_x ^{\mathcal{E}}$ is the uniform measure on $\mathcal{D}_{m^k}$ for every $k$. By the Kolmogorov consistency Theorem, $\alpha_x ^{\mathcal{E}} = \lambda$ almost surely. Since $\alpha = \int \alpha_x ^{\mathcal{E}} d\rho (x)$, this proves the result.
\end{proof}

\begin{Claim} \label{Claim unqiue} 
Let $\theta \in \mathcal{P}(\mathbb{T}^2)$ be a $T_m \times T_n$ invariant measure such that $P_2 \theta = \rho$ is exact dimensional. If $\dim \theta =1+\dim \rho$ then $\theta = \lambda\times \rho$.
\end{Claim}
\begin{proof}
By equation  \eqref{Eq lower dim}, and by Theorem \ref{Theorem Led young}
\begin{eqnarray} \label{Eq 1}
\notag 1+ \dim \rho &=& \dim \theta \\
\notag &= &\text{ess-inf}_{x\sim \theta} \dim (\theta,x)\\
\notag &=& \text{ess-inf}_{x\sim \theta} \frac{h(\theta^x, T_m\times T_n)-h(( P_2 \theta ) ^{P_2  x}, T_n)}{\log m}+\frac{ h(( P_2 \theta ) ^{P_2  x}, T_n)}{\log n}\\
  &=& \text{ess-inf}_{x\sim \theta} \frac{h(\theta^x, T_m\times T_n)-h(\rho ^{P_2  x}, T_n)}{\log m}+\frac{ h(\rho ^{P_2  x}, T_n)}{\log n}
\end{eqnarray}
(recall that $\theta^x$ and $\rho^{P_2 x}$ denote the corresponding ergodic components of $\theta$ and of $\rho$, respectively).

Now,  by  Theorem \ref{Theorem 9.1}, and since $\rho$ has exact dimension
\begin{equation*} 
\text{ess-sup}_{x\sim \theta} \dim \left( \rho ^{ P_2x} \right) = \text{ess-sup}_{y \sim \rho} \dim \left( \rho ^{y} \right) = \overline{\dim} \rho = \dim \rho = \text{ess-inf}_{x \sim \theta} \dim \left( \rho ^{P_2 x} \right).
\end{equation*}
So, for $\theta$ almost every $x$ we have
\begin{equation} \label{Eq 2}
\frac{ h(\rho ^{P_2  x}, T_n)}{\log n} = \dim \rho^{P_2 x}  = \dim \rho  .
\end{equation}
Combining \eqref{Eq 2} with \eqref{Eq 1}, we find that
\begin{equation} \label{Eq 3}
1= \text{ess-inf}_{x\sim \theta} \frac{h(\theta^x, T_m\times T_n)-h(\rho ^{P_2  x}, T_n)}{\log m}.
\end{equation}
Therefore, by \eqref{Eq 3}, the formula for entropy as an average over ergodic components, the Abramov-Rokhlin Lemma, and the formula for entropy as the average of conditional measures (as in Lemma \ref{Lemma entropy}), we have
\begin{eqnarray*}
\log m &\leq & \int \left( h(\theta^x, T_m\times T_n)-h(\rho ^{ P_2x}, T_n) \right) d \theta (x)\\
 &=& \int h(\theta^x, T_m\times T_n)d \theta (x) - \int h(\rho ^{ P_2x}, T_n) d \theta (x)\\
&=& h(\theta,T_m \times T_n) -\int h(\rho ^{ y}, T_n) d \rho (y)\\
&=& h(\theta,T_m \times T_n) -h(\rho,T_n)\\
&=& h(\theta, T_m \times T_n | \mathcal{E}) \\
&\leq & \log m
\end{eqnarray*}
where $\mathcal{E}$ be the invariant sigma algebra that corresponds to the second coordinate of $\mathbb{T}^2$. Thus, we have that $\theta$ almost surely,
\begin{equation} \label{Eq 4}
\frac{h(\theta^x, T_m\times T_n)-h(\rho ^{ P_2x}, T_n)}{\log m} =1, \quad \text{ and } h(\theta, T_m \times T_n | \mathcal{E})=\log m.
\end{equation}

Now, \eqref{Eq 4} and the Abramov-Rokhlin Lemma imply that $\theta$ almost surely
\begin{equation} \label{Eq 5}
\log m + h(\rho ^{ P_2x}, T_n) = h(\theta^x, T_m\times T_n) = h(\theta^x, T_m \times T_n | \mathcal{E}) + h( P_2 \theta ^x, T_n).
\end{equation}
By \eqref{Eq 4} and the formula for entropy and convex combinations,
\begin{equation*}
\log m = h(\theta, T_m \times T_n | \mathcal{E}) = \int h(\theta^x, T_m \times T_n | \mathcal{E}) d \theta (x).
\end{equation*}
Since $0\leq h(\theta^x, T_m \times T_n | \mathcal{E}) \leq \log m$ almost surely, we must have $h(\theta^x, T_m \times T_n | \mathcal{E}) = \log m$ almost surely. By this equality and \eqref{Eq 5} we see that for $\theta$ almost every $x$,
\begin{equation} \label{Eq 7}
h( P_2 \theta ^x, T_n)=  h(\rho ^{ P_2x}, T_n).
\end{equation}

Finally, by \eqref{Eq 7} and \eqref{Eq 5}, 
\begin{equation*}
 h(\theta^x, T_m \times T_n) = \log m +h(P_2 \theta^x, T_n) \quad \text{ for almost every x}.
\end{equation*}
By Lemma \ref{Lemma entropy}, almost every ergodic component $\theta^x$  equals $\lambda\times P_2 \theta^x$. Thus,
\begin{equation*}
\theta = \int \theta^x d \theta (x) = \int \lambda \times P_2 \theta^x  d \theta (x) = \lambda \times \left( \int P_2 \theta^x d\theta (x) \right) = \lambda \times P_2 \left( \int \theta^x d\theta(x) \right)=  \lambda\times \rho. 
\end{equation*}
\end{proof}

Next, we make a short digression to discuss the relation between the conditional measures of a convolution of measures, and the conditional measures of the individual measures convolved, in some special cases.  In the following, the convolution of the two measures on the unit square $[0,1]^2$ is understood to take place in $\mathbb{R}^2$. For a measure $\nu \in \mathcal{P}(\mathbb{R}^2)$, Let $\nu = \int \nu_y dP_2 \nu (y)$ be the disintegration of $\nu$ with respect to the projection $P_2$.

\begin{Claim} \label{Claim conditionals} 
Let $\theta,\nu \in \mathcal{P}([0,1]^2)$ be two measure such that $\theta = \tau \times \alpha$, where the measure $\alpha$ is a convex combination of finitely many atomic measures. Then for $P_2 (\nu*\theta)$ almost every $z$, the conditional measure $(\nu*\theta)_z$ with respect to the projection $P_2$ is a finite convex combination of measures of the form $\nu_{z-z_i}*(\tau\times \delta_{z_i})$, where $z_i$ is an atom of $\alpha$ and $\nu_{z-z_i}$ is a conditional measure of $\nu$ with respect to the projection $P_2$. 
\end{Claim}

\begin{proof}
If $\alpha = \delta_y$ for some $y$ then the result is straightforward. For the general case, notice that if $\theta = \tau\times \alpha$ and $\alpha = \sum p_i \delta_{z_i}$ then by the linearity of both convolution and of taking product measures
\begin{equation*}
\nu * \theta = \sum \nu*(\tau \times (p_i \cdot \delta_{z_i})) = \sum p_i \cdot \nu*(\tau \times  \delta_{z_i}).
\end{equation*}
In general, if $\mu = \mu_1\cdot p_1 + \mu_2\cdot p_2$ is a convex combination of probability measures and $\mathcal{C}$ is some sigma algebra,  then the following holds almost surely for every $f\in L^1$:
\begin{equation*}
\mathbb{E}_{\mu} (f|\mathcal{C}) = p_1 \cdot \mathbb{E}_{\mu_1}(f|\mathcal{C})\cdot \frac{d \mu_1}{d \mu}+p_2 \cdot \mathbb{E}_{\mu_2} (f|\mathcal{C})\cdot \frac{d \mu_2}{d \mu}.
\end{equation*}
We remark that in the above equation,  the Radon-Nikodym derivatives $\frac{d \mu_i}{d \mu}$ in fact stand for the Radon-Nikodym derivatives when the measures are restricted to the sigma-algebra $\mathcal{C}$, i.e. $\frac{d \mu_i|_\mathcal{C}}{d \mu|_{\mathcal{C}}}$. However, we suppress this in our notation. So, for $\mathcal{B}_2$ the Borel sigma algebra on the $y$-axis, for every $f\in L^1$ and  for almost every $z$
\begin{eqnarray*}
\int f d(\nu * \theta)_z &=& \mathbb{E}_{\nu * \theta } (f|P_2 ^{-1} \mathcal{B}_2)(z) \\
& = &\sum_i p_i \cdot \mathbb{E}_{\nu*(\tau \times  \delta_{z_i})}(f|P_2 ^{-1} \mathcal{B}_2) (z)\cdot \frac{d \nu*(\tau \times \delta_{z_i})}{d \nu * \theta} (z) \\
& = & \sum_i p_i \cdot  \int f d \left( \nu*(\tau \times  \delta_{z_i}\right))_z \cdot \frac{d \nu*(\tau \times \delta_{z_i})}{d \nu * \theta} (z) \\
& = & \sum_i p_i \cdot  \int f d \left( \nu_{z-z_i}*(\tau \times  \delta_{z_i}) \right) \cdot \frac{d \nu*(\tau \times  \delta_{z_i})}{d \nu * \theta} (z) \\
& = &  \int f d \left( \sum_i  \left( \nu_{z-z_i}*(\tau \times \delta_{z_i}) \right)\cdot p_i \cdot \frac{d \nu*(\tau \times \delta_{z_i})}{d \nu * \theta} (z)  \right)
\end{eqnarray*}
It follows that almost surely,
\begin{eqnarray*}
(\nu * \theta)_z &=& \sum_i \nu_{z-y_i}*(\tau \times  \delta_{z_i})\cdot p_i \cdot  \frac{d \nu*(\tau \times  \delta_{z_i})}{d \nu * \theta} (z)  
\end{eqnarray*}
\end{proof}

The following Claim, which forms the main result of this section, is also the key for our argument. 
\begin{Claim}  \label{Claim rigidty}
Let $\nu \in \mathcal{P}([0,1]^2)$ be a $T_m \times T_n$ invariant measure such that:
\begin{enumerate}
\item We have $P_2 \nu = \rho$, where $\rho$ is a continuous ergodic measure, and  $P_1\nu = \lambda$.

\item There exists some $\delta>0$ such that:

 For every probability measure $\tau\in \mathcal{P}([0,1])$  with $\dim \tau \geq 1- \delta$, for $ \rho$ almost every $y$, we have $\dim \tau*P_1 \nu_y =1$.
\end{enumerate}
Then $\nu = \lambda\times \rho$.
\end{Claim}
\begin{proof}
Suppose towards a contradiction that $\nu \neq \lambda\times \rho$. Let us first identify $\nu$ with the corresponding measure on $\mathbb{T}^2$ (i.e. we project $\nu$ to $\mathbb{T}^2$ but we keep the notation $\nu$), which cannot be $\lambda\times \rho$ either.  Then, by Lemma \ref{Lemma fou}, there exists $(i,j)\in \mathbb{Z}^2 \setminus \lbrace (0,0) \rbrace$ such that 
$$\hat{\nu}(i,j)\neq \widehat{\lambda\times \rho} (i,j).$$
Now, as $P_2 \nu =  \rho$ and  $P_1 \nu=\lambda$ we must have $i,j\neq 0$, since if e.g. $i=0$ then, using \eqref{Eq fouier prod},
$$ \hat{\nu}(0,j) = \widehat{P_2 \nu}(j) = \hat{\rho}(j) =1\cdot \hat{\rho}(j) = \hat{\lambda}(0)\hat{\rho}(j)=\widehat{\lambda\times \rho} (0,j)$$
a contradiction. Thus, we may assume both $i,j\neq 0$, and since $\hat{\lambda} (i)=0$ we have $\hat{\nu}(i,j)\neq 0$ by \eqref{Eq fouier prod}.

Now, let $k\in \mathbb{N}$ be such that $2|j|<n^k+1$. We construct two measures $\tau, \alpha \in \mathcal{P}(\mathbb{T})$  such that:
\begin{enumerate}
\item  The measure $\alpha$ is a uniform measure on a finite (periodic) $T_{n^k}$ orbit such that $\hat{\alpha}(j)\neq 0$.

To find such a measure, we take the $T_{n^k}$ periodic orbit $\lbrace x_0,x_1 \rbrace$ where $x_0 = \frac{1}{n^{2k} -1}$ and $x_1 = \frac{n^k}{n^{2k}-1}$. Define a measure $\alpha = \frac{1}{2}\delta_{x_0} + \frac{1}{2}\delta_{x_1}$ on this orbit. Then
\begin{equation*}
\hat{\alpha}(x) = \frac{1}{2}\cdot e^{2 \pi i \frac{x}{n^{2k}-1}} ( 1+ e^{2 \pi i \frac{x(n^k-1)}{n^{2k}-1}})
\end{equation*}
Now, if $\hat{\alpha}(j)=0$ then $e^{2 \pi i \frac{j(n^k-1)}{n^{2k}-1}} =-1$, which can only happen if $2 \frac{j(n^k-1)}{n^{2k}-1} \in  1 + 2\mathbb{Z}$. However, $2 \frac{j(n^k-1)}{n^{2k}-1} = 2  \frac{j}{n^{k}+1}$ and $|\frac{2j}{n^k +1}|<1$. Thus, it is impossible that $\hat{\alpha}(j)=0$.

\item The measure $\tau$ is $T_{m^k}$ invariant, $\dim \tau \geq 1- \delta$ and $\hat{\tau} (i)\neq 0$.

To find such a measure, let $\beta$ be the Cantor-Lebesgue measure with respect to base $m$ and the non-degenerate probability vector $(\frac{1}{3},\frac{2}{3},0,...,0)$ (see the end of Section \ref{Section dimension}).  Then $\beta$ is a weakly mixing $T_m$ invariant measure (a Bernoulli measure). By \eqref{Eq fouier cantor leb},
\begin{eqnarray*}
\hat{\beta} (x) &=& \prod_{j=1} ^\infty ( \frac{1}{3}+\frac{2}{3}\exp ( 2 \pi i x \frac{1}{m^j})) \\
&=& \prod_{j=1} ^\infty \left( 1+ \frac{2}{3}\cdot (\exp ( 2 \pi i x \frac{1}{m^j})-1) \right)
\end{eqnarray*}
By looking at the corresponding power series expansion, we see that for every $j$
\begin{equation*}
|\exp ( 2 \pi i x \frac{1}{m^j}) -1| = |\sum_{k=1} ^\infty \frac{(2 \pi i x \frac{1}{m^j})^k}{k!}| \leq \sum_{k=1} ^\infty \frac{(2\pi \frac{|x|}{m^j})^k}{k!} \leq \frac{1}{m^j}\cdot  \exp(2\pi |x|).
\end{equation*}
By Proposition 3.1 in Chapter 5 of \cite{Stein2003complex}, we conclude that $\hat{\beta} (i)=0$ if and only if one of its factors has a zero at $i$. Since $0\neq i\in \mathbb{Z}\subset \mathbb{R}$, this clearly does not happen, and we conclude that $\hat{\beta} (i) \neq 0$.

Also, notice that $\dim \beta = \frac{H(\frac{1}{3},\frac{2}{3})}{\log m} >0$, where $H(p_1,p_2)$ is the Shannon entropy of the probability vector $(p_1,p_2)$. Finally, by Theorem \ref{Theorem Elon conv}, there exists some $q\in \mathbb{N}$ such that $\dim \beta^{*q} > 1-\delta$, where by $\beta^{*q}$ we mean that we convolve $\beta$ with itself $q$ times. Recalling \eqref{Eq fourier conv}, we see that $\widehat{\beta^{*q}}(i)=\left( \hat{\beta}  (i) \right)^q \neq 0$. Thus, we may take $\tau = \beta^{*q}$. Notice that $\tau$ is $T_m$ invariant, so it is also $T_{m^k}$ invariant. 
\end{enumerate}

Thus, by \eqref{Eq fourier conv} and \eqref{Eq fouier prod},  the Fourier coefficients of the measure $\nu*(\tau\times \alpha) \in \mathcal{P}(\mathbb{T}^2)$ satisfy
\begin{equation*}
\widehat{\left( \nu*(\tau\times \alpha) \right)} (i,j) = \hat{\nu}(i,j)\cdot \widehat{(\tau \times \alpha)}(i,j) = \hat{\nu}(i,j)\cdot \hat{\tau}(i) \cdot \hat{\alpha}(j)\neq 0.
\end{equation*}
Therefore, as $i\neq0$, we have by Lemma \ref{Lemma fou}
\begin{equation} \label{Eq contra}
\nu*(\tau\times \alpha) \neq \lambda\times (\rho*\alpha),
\end{equation}
since 
$$\widehat{\left( \lambda\times (\rho*\alpha) \right)}(i,j) = \hat{\lambda}(i)\cdot \widehat{\rho*\alpha}(j)=  0.$$

On the other hand, let us now  lift all our measures to corresponding measures on $[0,1]$ and $[0,1]^2$. Since $\nu$ is already defined on the unit square, we take this representative for our lift. Since $\tau$ cannot be atomic we can take our lift as the corresponding measure on $[0,1]$, and for the measure $\alpha$ we can take essentially the same measure.  By Claim \ref{Claim conditionals}, the conditional measures of $\nu*(\tau\times \alpha)$ with respect to the projection $P_2$  are almost surely finite convex combinations of measures of the form $\nu_{y-x_i} *(\tau\times \delta_{x_i})$, where $i=0,1$ are the atoms of the measure $\alpha$, with weights $p_i (y)$ for $i=0,1$. So, for $P_2 \left( \nu*(\tau\times \alpha) \right)$ almost every $y$,
\begin{eqnarray*}
\dim \left( \nu*(\tau\times \alpha) \right)_y &=& \dim \left( \nu_{y-x_0} *(\tau\times \delta_{x_0})\cdot p_0 (y)+ \nu_{y-x_1} *(\tau\times \delta_{x_1})\cdot p_1 (y) \right) \\
&=& \min \lbrace \dim \nu_{y-x_0} *(\tau\times \delta_{x_0})\cdot p_0 (y), \quad \dim \nu_{y-x_1} *(\tau\times \delta_{x_1})\cdot p_1 (y) \rbrace \\
&=& \min \lbrace \dim \nu_{y-x_0} *(\tau\times \delta_{x_0}), \quad \dim \nu_{y-x_1} *(\tau\times \delta_{x_1}) \rbrace \\
&\geq & \min \lbrace \dim P_1 \left( \nu_{y-x_0} *(\tau\times \delta_{x_0}) \right), \quad \dim P_1 \left( \nu_{y-x_1} *(\tau\times \delta_{x_1}) \right) \rbrace \\
&\geq & \min \lbrace \dim \left( P_1  \nu_{y-x_0} \right) *\tau, \quad \dim \left( P_1  \nu_{y-x_1} \right) *\tau \rbrace \\
&= & 1 \\
\end{eqnarray*} 
where we have used condition (2) in the statement of the Claim, the lower bound on $\dim \tau$ and  Proposition \ref{Prop properties of dim}. Since the opposite inequality is always true, we conclude that
\begin{equation} \label{Eq condiotnals}
\dim \left( \nu*(\tau\times \alpha) \right)_y = 1, \text{ for } P_2 (\nu*(\tau\times \alpha)) \text{ almost every } y.
\end{equation}

Since $P_2 (\nu*(\tau\times \alpha)) = \rho *\alpha$, and by \eqref{Eq condiotnals}, we see via Lemma \ref{Lemma 3.5} part (1) that 
$$\dim \nu*(\tau\times \alpha)\geq 1+\dim \rho*\alpha.$$
 On the other hand, by  part (2) of Lemma \ref{Lemma 3.5}, and since $P_1 \nu = \lambda$, 
$$\dim \nu*(\tau\times \alpha)\leq \dim_p \lambda*\tau +\dim \rho*\alpha = 1+\dim \rho*\alpha.$$ 
We conclude that $\dim \nu*(\tau\times \alpha) = 1+\dim \rho*\alpha$.

Finally, we project $\nu*(\tau\times \alpha)$ to $\mathbb{T}^2$. Since this projection is a local diffeomorphism, it preserves dimension. Thus, the convolved measure $\nu*(\tau\times \alpha)$, with the ambient group being $\mathbb{T}^2$, has dimension $1+\dim\rho*\alpha$. Moreover, by Theorem \ref{Theorem 9.1}, since $\rho$ is ergodic it is exact dimensional. Since $\alpha$ is a discrete measure (supported on two atoms), the convolution $\rho * \alpha$ remains exact dimensional (Proposition \ref{Prop properties of dim}).

 Therefore, we may apply Claim \ref{Claim unqiue} for the measure $\nu*(\tau\times \alpha)$, since this is a $T_{m^k} \times T_{n^k}$ invariant measure (as the convolution of such measures), and the assumptions on the dimension of $\nu*(\tau\times \alpha)$ and on $P_2 (\nu*(\tau\times \alpha))=\rho *\alpha$  are met by the previous paragraph. Thus, we may conclude that $\nu*(\tau\times \alpha) = \lambda\times (\rho*\alpha) $. Via \eqref{Eq contra}, this yields our desired contradiction.
\end{proof}

\section{Proof  of Theorem \ref{Conjecture}} \label{Section rep}

Let $\mu$ be as in Theorem \ref{Conjecture}, and let $\nu$ be some accumulation point of the  sequence of measures as in \eqref{Eq I1} (where we pick a typical $x$ according to $\mu$), along a subsequence $N_k$. Our goal is to show that $\nu=\lambda\times \rho$, and we shall do this by showing that $\nu$ meets the conditions of Claim \ref{Claim rigidty}.

By our assumptions and  Theorem 1.1 in \cite{hochmanshmerkin2015}  it follows that $P_1 \nu = \lambda$ and $P_2 \nu = \rho$. Thus, $\nu$ satisfies condition (1) in Claim \ref{Claim rigidty}. Notice that this implies that $\nu$ gives zero mass to the points of discontinuouty of $T_m \times T_n$.  So, $\nu$ is $T_m \times T_n$ invariant. For the second condition of Claim \ref{Claim rigidty}, we require the following analogue of Theorem 5.1 in \cite{hochmanshmerkin2015}. Recall that $P$ is the EFD generated by $\mu$ (see Section \ref{Section EFD}).

\begin{Claim} \label{Theorem 5.1} (Conditional integral representation)
For $P_2 \nu = \rho$ almost every $y$ there is a probability space $(\Omega, \mathcal{F}, Q(y))$ and measurable functions 
$$c:\Omega\rightarrow (0,\infty),\quad x:\Omega \rightarrow [-1,1], \quad  \eta:\Omega \rightarrow \mathcal{P}([-1,1])$$
such that:
\begin{enumerate}
\item \begin{equation*}
P_1 \nu_y = \int c (\omega) \cdot (\delta_{x (\omega)}*\eta (\omega))|_{[0,1]} dQ(y)(\omega)
\end{equation*}

\item Let $P_y$ denote the distribution of random variable $\eta$ as above. Then $P = \int P_y d \rho(y)$.
\end{enumerate}

\end{Claim}

\begin{proof}
We dedicate the first part of the proof to finding a disintegration of $P$ according to the measure $\rho$. To this end, consider the following sequence of distributions $R_{N_k}\in \mathcal{P}( \mathcal{P}([0,1])\times [0,1])$, defined by
\begin{equation*}
R_{N_k} = \frac{1}{N_k} \sum_{i=0} ^{N_k-1} \delta_{(\mu_{x,i\log m}, T_n ^i x)}.
\end{equation*}
Let $R$ be some accumulation point of this sequence. Without the loss of generality, let us assume the limit already exists along the sequence $N_k$.  Then we may assume that $P_1 R = P$ and $P_2 R=\rho$, since we are considering a $\mu$ typical point $x$, making use of the fact that $\mu$ is pointwise generic under $T_n$ for $\rho$, and of the spectral condition on $P$ via Proposition \ref{Prop spec assumption}.

Next, we disintegrate the distribution $R$ via the projection $P_2$:
\begin{equation*}
R = \int R_y d\rho(y).
\end{equation*}
Applying the map $P_1$ to this disintegration, we see that
\begin{equation*}
P= P_1 R = \int P_1 R_y d \rho(y).
\end{equation*}
Thus, the family of measures $\lbrace P_1 R_y \rbrace$ forms our desired disintegration. 

Let us study this family of distributions a little further: It is well known (see e.g. \cite{furstenberg2008ergodic} or \cite{Simmons2012cond}) that for $\rho$ almost every $y$, we may write
\begin{equation*}
R_y = \lim_{p\rightarrow \infty} \frac{R(\cdot \cap P_2 ^{-1} (\mathcal{D}_{2^p} (y)) }{\rho(\mathcal{D}_{2^p} (y))}.
\end{equation*}
Therefore,
\begin{equation*}
P_1 R_y = \lim_{p\rightarrow \infty} \frac{P_1 R(\cdot \cap P_2 ^{-1} (\mathcal{D}_{2^p} (y)) }{\rho(\mathcal{D}_{2^p} (y))} = \lim_{p\rightarrow \infty} \lim_{k\rightarrow \infty} \frac{1}{\rho(\mathcal{D}_{2^p} (y))\cdot N_k} \sum_{\lbrace 0\leq i \leq N_k-1:\quad T_n ^i x \in \mathcal{D}_{2^p} (y)\rbrace} \delta_{\mu_{x,i\log m}}.
\end{equation*}
Finally, we note that for $\rho$ almost every $y$, for every $p$,
\begin{equation*}
\frac{R(\cdot \cap P_2 ^{-1} (\mathcal{D}_{2^p} (y)) }{\rho(\mathcal{D}_{2^p} (y))} \ll R.
\end{equation*}
Therefore, for $\rho$ almost every $y$, for every $p$,
\begin{equation} \label{Eq abs con}
\frac{P_1R(\cdot \cap P_2 ^{-1} (\mathcal{D}_{2^p} (y)) }{\rho(\mathcal{D}_{2^p} (y))} \ll P_1R=P.
\end{equation}

We now turn our attention to the main assertions of the Claim. First, we embed $\mu$ (the measure from Theorem \ref{Conjecture}) on the diagonal of the unit square by pushing it forward via the map $x\mapsto (x,x)$. We call this new measure $\tilde{\mu}$.  For $k\in \mathbb{N}$ let $\mathcal{A}^k$ denote the partition of $[0,1]^2$ given by 
$$\mathcal{D}_{m^k} \times \mathcal{D}_{n^k} = \bigvee_{i=0} ^{k-1} (T_m \times T_n)^i \left( \mathcal{D}_{m} \times \mathcal{D}_{n} \right) .$$
Given a point $z \in [0,1]^2$ such that $\tilde{\mu} (\mathcal{A}^k (z))>0$ we define a probability measure
$$\tilde{\mu}_{\mathcal{A}^k (z)} : = c \cdot (T_m \times T_n)^k (\tilde{\mu}|_{\mathcal{A}^k (z)}),$$
where $c$ is a normalizing constant. By applying Claim \ref{Claim stable intersections}, we see that there is a set $S\subseteq \mathbb{N}$ of density $1$ (possibly depending on the $x$ we chose according to $\mu$), such that for every $k\in S$ the measure $\tilde{\mu}|_{\mathcal{A}^k (x,x)}$ is an affine image of the measure $\mu|_{\mathcal{D}_{m^k} (x)}$. Since we are only interested in the  limiting behaviour of these measures, we may assume $S=\mathbb{N}$. Also, $\nu (\partial A)=0$ for all $A\in \mathcal{A}^k,k\in \mathbb{N}$ since $P_1 \nu$ and $P_2 \nu$ are both continuous measures. Thus, by Theorem  \ref{Theorem 2.1}
\begin{equation} \label{Eq I2}
\nu = \lim_{k\rightarrow \infty} \frac{1}{N_k} \sum_{i=0} ^{N_k-1} \tilde{\mu}_{\mathcal{A}^i (x,x)}
\end{equation}

Now, for $P_2 \nu = \rho$ almost every $y$ the conditional measure $\nu_y$ can be obtained as the weak-* limit $\lim_{p \rightarrow \infty} \nu_{P_2 ^{-1} \mathcal{D}_{2^p} (y)}$, where for every Borel set $A\subset [0,1]^2$ and $p\in \mathbb{N}$,
\begin{equation*}
\nu_{P_2 ^{-1} \mathcal{D}_{2^p} (y)} (A) := \frac{\nu(A \cap P_2 ^{-1}\mathcal{D}_{2^p} (y))}{\nu(P_2 ^{-1} \mathcal{D}_{2^p} (y))} = \frac{\nu(A \cap P_2 ^{-1}\mathcal{D}_{2^p} (y))}{\rho (\mathcal{D}_{2^p} (y))}.
\end{equation*}

Fix $p\in \mathbb{N}$. By \eqref{Eq I2} and since $\tilde{\mu}|_{\mathcal{A}^k (x,x)}$ is an affine image of the measure $\mu|_{\mathcal{D}_{m^k} (x)}$ for every $k$,   the projection of $\nu_{P_2 ^{-1} \mathcal{D}_{2^p} (y)}$ to the $x$-axis (i.e. via $P_1$) equals\footnote{Recall that by equation (8) in \cite{hochmanshmerkin2015} Section 5.2,  
$$\mu_{\mathcal{D}_{m^i} (x)}:=  c \cdot T_m ^k (\mu|_{\mathcal{D}_{m^k} (z)}) =c_k \left( \tau_{x_k}*  \mu_{x,i\log m} \right)|_{[0,1]}$$
for corresponding parameters.}
\begin{equation*}
 \lim_{k\rightarrow \infty} \frac{1}{\rho (\mathcal{D}_{2^p} (y)) \cdot N_k} \sum_{\lbrace i: 0\leq i \leq N_k-1,  \text{ and } T_n ^i (x)\in \mathcal{D}_{2^p} (y)\rbrace} P_1 \circ \left( L_{\frac{n^i}{m^i}, (T_m ^i (x), T_n ^i (x))} c_k \left( \tau_{x_k}*  \mu_{x,i\log m} \right)|_{[0,1]}  \right)
\end{equation*}
\begin{equation} \label{Eq analogue}
= \lim_{k\rightarrow \infty} \frac{1}{\rho (\mathcal{D}_{2^p} (y)) \cdot N_k} \sum_{\lbrace i: 0\leq i \leq N_k-1, \quad \text{ and } T_n ^i (x)\in \mathcal{D}_{2^p} (y)\rbrace}  c_k \left( \tau_{x_k}*  \mu_{x,i\log m} \right)|_{[0,1]}  
\end{equation}
where $L_{\alpha,z}$ is the unique affine map taking the $x$-axis to the line with slope $\alpha$ through the point $z$. Notice that in the first equation above we only take note of the indices such that $T_n ^i (x)\in \mathcal{D}_{2^p} (y)$, and this is justified by  Lemma \ref{Lemma bad indices}. 

We thus see, as in Theorem 5.1 in \cite{hochmanshmerkin2015} and its preceding discussion, that    there is a   distribution 
\begin{equation*}
Q_{\mathcal{D}_{2^p} (y)} \in \mathcal{P} \biggl( \mathbb{R} \times [-1,1] \times \mathcal{P}(\mathcal{P}([0,1])) \biggr)
\end{equation*}
such that we have an integral representation (that depends on both $p$ and $y$)
\begin{equation*}
P_1 \nu_{P_2 ^{-1} \mathcal{D}_{2^p} (y)} = \int g(\omega) dQ_{\mathcal{D}_{2^p} (y)} (\omega)
\end{equation*}
where $g:\mathbb{R} \times [-1,1] \times \mathcal{P}(\mathcal{P}([0,1])) \rightarrow \mathcal{P}([0,1])$ is the map $g(c,x,\eta) = c\cdot (\delta_x * \eta)|_{[0,1]}$. Moreover, the  distribution of $Q_{\mathcal{D}_{2^p} (y)}$ on the measure component $\mathcal{P}(\mathcal{P}([0,1]))$ is given by 
\begin{equation*}
\lim_{k\rightarrow \infty} \frac{1}{\rho(\mathcal{D}_{2^p} (y))\cdot N_k} \sum_{\lbrace 0\leq i \leq N_k-1:\quad T_n ^i x \in \mathcal{D}_{2^p} (y)\rbrace} \delta_{\mu_{x,i\log m}},
\end{equation*}
and by equation \eqref{Eq abs con} and its preceding discussion, this distribution is absolutely continuous with respect to $P$.

Notice that for $Q_{\mathcal{D}_{2^p} (y)}$ almost every $(c,x,\eta)$, $c$ is the normalizing constant making $(x*\eta)|_{[0,1]}$ a probabilty measure.  Also, the map $g$ is continuous almost surely. Moreover, by moving to a subsequence, we may assume the weak -* limit $\lim_{p\rightarrow \infty} Q_{\mathcal{D}_{2^p} (y)}$ exists, call it $Q_y$. For these assertions, we argue, as in (\cite{hochmanshmerkin2015}, Theorem 5.1),  that the distribution $\lbrace  P_1 (Q_{\mathcal{D}_{2^p} (y)}) \rbrace_p$ of the normalizing constants is tight. Indeed, for measures drawn according to $P$ this follows from  Proposition \ref{Prop tight}, and in our case the  distribution of $Q_{\mathcal{D}_{2^p} (y)}$ on the measure component is absolutely continuous with respect to $P$. Finally,
\begin{equation*}
P_1 \nu_y = P_1 \lim_{p\rightarrow \infty}  \nu_{P_2 ^{-1} \mathcal{D}_{2^p} (y)} =  \lim_{p\rightarrow \infty} P_1 \nu_{P_2 ^{-1} \mathcal{D}_{2^p} (y)} = \lim_{p\rightarrow \infty} \int g(\omega) dQ_{\mathcal{D}_{2^p} (y)} (\omega) = \int g(\omega) dQ_y(\omega).
\end{equation*} 
This completes the proof of part (1). For part (2), it remains to note that by our construction, for $\rho$ almost every $y$ the distribution of $Q_y$ on the measure component $\mathcal{P}(\mathcal{P}([0,1]))$ is given by $P_1 R_y$, as in the first part of the proof, by the discussion preceding \eqref{Eq abs con}.
\end{proof}

\textbf{Proof of Theorem \ref{Conjecture}} We are now in position to show that $\nu$ satisfies all conditions in Claim \ref{Claim rigidty}. We have already established that condition (1) holds. As for condition (2), we choose $\delta = \dim P>0$ (here we use the assumption that $P$ is non trivial, and Proposition \ref{Prop posit dim for EFD}). Let $\tau\in \mathcal{P}([0,1])$ be such that $\dim \tau \geq 1-\delta$. We now show that $\dim \tau*P_1\nu_y =1$ for $P_2 \nu = \rho$ almost every $y$ץ

First, by  Lemma \ref{Lemma 5.8} and Claim \ref{Theorem 5.1} part (2) 
\begin{equation} \label{Eq Lemma 5.8}
1=  \int \dim (\tau*\eta) d P(\eta) =\int \dim (\tau*\eta) dP_y (\eta) d \rho (y).
\end{equation} 
Therefore, for $\rho$ almost every $y$, for $P_y$ almost every $\eta$, $\dim \tau*\eta =1$ (since the integrand is always $\leq 1$). Thus, by Claim \ref{Theorem 5.1} part (1), for $\rho$ almost every $y$,
\begin{eqnarray*} 
\dim \tau*P_1 \nu_y &=& \dim \tau * \int c (\omega) \cdot (\delta_{x (\omega)}*\eta (\omega))|_{[0,1]} dQ(y)(\omega)\\
&=&  \dim \int c (\omega) \cdot \tau*( \delta_{x (\omega)}*\eta (\omega))|_{[0,1]} dQ(y)(\omega)\\
&\geq& \text{ess-inf}_{\omega\sim Q(y)} \dim \tau*( \delta_{x (\omega)}*\eta (\omega))|_{[0,1]}\\
&\geq& \text{ess-inf}_{\omega\sim Q(y)} \dim \tau* \delta_{x (\omega)}*\eta (\omega)\\
&=& \text{ess-inf}_{\eta\sim P_y} \dim \tau*\eta\\
&=& 1\\
\end{eqnarray*}
Since $\dim \tau*P_1\nu_y\leq 1$ is always true, we find that $\dim \tau*P_1\nu_y= 1$ for $\rho$ almost every $y$.

We conclude that $\nu$ satisfies the conditions of Claim \ref{Claim rigidty}. Therefore, $\nu= \lambda\times \rho$, as desired.

\section{Proof of Theorem \ref{Main Theorem}} \label{Section final proof}
Let $\mu$ be a $T_p$-invariant and ergodic measure with positive dimension. Then $\mu$ generates an EFD $P$ with $\dim P>0$ by Theorem \ref{Theorem 2.4}.  Let $m>n>1$. The pure point spectrum  $\Sigma(P,S)$ can contain non zero integer multiplies of $\frac{1}{\log m}$ only if either $m\sim p$ (in Theorem \ref{Main Theorem} we assume this is not the case), or if $\frac{\log p}{\log m}\in \Sigma(T,\mu)$, see \cite{hochman2010geometric}. We shall prove Theorem \ref{Main Theorem} by using  Theorem \ref{Conjecture}, and following the analysis of Hochman and Shmerkin from (\cite{hochmanshmerkin2015}, Section 8) in order to relax the spectral condition (i.e. deal with the latter case). We begin by treating the case $n=p$.

Suppose first that $\Sigma(P,S)$ does not contain a non-zero integer multiple of $\frac{1}{\log m}$. By the ergodic Theorem, $\mu$ is pointwise $T_n$ generic for $\mu$.  Also, since $\mu$ generates an EFD such that $\frac{k}{\log m} \notin \Sigma (P,S)$ for every $k\in \mathbb{Z}\setminus \lbrace 0 \rbrace$, we may apply Theorem \ref{Conjecture} and obtain 
\begin{equation*} 
\frac{1}{N} \sum_{i=0} ^{N-1} \delta_{(T_m ^i (x) , T_n ^i (x))} \rightarrow \lambda \times \mu, \quad \text{ for } \mu \text{ almost every } x.
\end{equation*} 
.

Suppose now that there exists some $k\in \mathbb{Z}\setminus \lbrace 0 \rbrace$ such that $\frac{k}{\log m}\in \Sigma(P,S)$, so   $P$ is not $S_{\log m}$ ergodic by Proposition 4.1 in \cite{hochmanshmerkin2015}. By the results discussed in  (\cite{hochmanshmerkin2015},  Sections 8.2 and 8.3) there is a probability space $(\Omega,\mathcal{F},Q)$ and a measurable family of measures $\lbrace \mu_\omega \rbrace_{\omega\in \Omega}$ such that:
\begin{enumerate}
\item The measures $\lbrace \mu_\omega \rbrace_{\omega\in \Omega}$ form a disintegration of $\mu$, that is, $\mu = \int \mu_\omega dQ(\omega)$.

\item For $Q$ almost every $\omega$, $\mu_\omega$ generates $P$.

\item  For $Q$ almost every $\omega$, $\mu_\omega$ $\log m$-generates an $S_{\log m}$ ergodic distribution $P_x$ at almost every point $x$ (see Proposition \ref{Prop spec assumption} for the definition of $\log m$-generation).
\end{enumerate}

Let $\delta=\dim P>0$ denote the almost sure dimension of the measures drawn by $P$. Then the following holds\footnote{Here, we use the fact that the commutative phase measure from Theorem 8.2 in \cite{hochmanshmerkin2015}  has dimension 1, as proven in Section 8.3.}:
\begin{Lemma} (\cite{hochmanshmerkin2015}, Lemma 8.3) \label{Lemma 8.3}
Let $\tau\in \mathcal{P}(\mathbb{R})$ be such that $\dim \tau \geq 1-\delta$. Then $\dim \tau*\eta=1$ for $Q$ almost every $\omega$, $\mu_\omega$ almost every $x$, and  $P_x$ almost every $\eta$.
\end{Lemma}

Now, we may finish the proof in a similar fashion to the proof of Theorem \ref{Conjecture}. Namely, For $Q$ almost every $\omega$ and  for $\mu_\omega$ almost every $x$, let $\nu$ be such that $(x,x)$ equidistribute for it sub-sequentially under $T_m\times T_n$. Then we may assume $P_2 \nu = \mu$ by the ergodic Theorem, and that $\mu_\omega$ $\log m$-generates $P_x$, where $P_x$ is typical with respect to Lemma \ref{Lemma 8.3}.  Then we have a conditional integral representation as in Claim \ref{Theorem 5.1}, but now we can only disintegrate $P_x = \int (P_x)_y d \mu (y)$. Since we have Lemma \ref{Lemma 8.3} at our disposal (so that an analogue of \eqref{Eq Lemma 5.8} holds for $P_x$ instead of $P$), we still have that for every $\tau \in \mathcal{P}(\mathbb{R})$ with $\dim \tau \geq 1-\delta$,   for $\mu$ almost every $y$, $\dim \tau*P_1\nu_y =1$ as the calculation carried out during the last stage of the proof of Theorem \ref{Conjecture} follows through in this case as well. It follows that $\nu = \lambda \times \mu$
. Finally, since this is true for $Q$ almost every $\mu_\omega$ and for $\mu_\omega$ almost every $x$, this is also true for $\mu$ almost every $x$ (recall that $\mu = \int \mu_\omega dQ(\omega)$).

The case when $n\not \sim p$ follows by a similar argument, only here for $Q$ almost every $\mu_\omega$, $\mu_\omega$ is pointwise $n$-normal, since this is true for $\mu$ by Theorem 1.10 in \cite{hochmanshmerkin2015}.

\section{Perturbing the initial point} \label{Section pert}
In this Section we prove the following generalization of Theorem \ref{Main Theorem}: 
\begin{theorem} \label{Main Theorem pert}
Let $\mu$ be a $T_p$ invariant ergodic measure with $\dim \mu>0$.  Let $m>n>1$ be integers such that $m\not \sim p$, and let $f,g\in \text{Aff}(\mathbb{R})$ be such that $f([0,1]), g([0,1])\subseteq [0,1]$. 
\begin{enumerate}
\item If $n=p$ then
\begin{equation*} 
\frac{1}{N} \sum_{i=0} ^{N-1} \delta_{(T_m ^i f(x) , T_n ^i x)} \rightarrow \lambda \times \mu, \quad \text{ for } \mu \text{ almost every } x,
\end{equation*}

\item If $n \not \sim p$ then
\begin{equation*} 
\frac{1}{N} \sum_{i=0} ^{N-1} \delta_{(T_m ^i f(x) , T_n ^i g(x))} \rightarrow \lambda\times \lambda, \quad \text{ for } \mu \text{ almost every } x,
\end{equation*} 
\end{enumerate}
\end{theorem}

The proof is similar to the proof of Theorem \ref{Main Theorem}. In particular, it relies on the following generalization of Theorem \ref{Conjecture}:
\begin{theorem} \label{Conjecture pert}
Let $\mu \in \mathcal{P}([0,1])$ be a probability measure, $f,g\in \text{Aff}(\mathbb{R})$, and $m>n>1$ be integers, such that:
\begin{enumerate}
\item The measure $\mu$ generates a non-trivial $S$-ergodic distribution $P\in \mathcal{P}(\mathcal{P}([-1,1]))$.

\item The pure point spectrum $\Sigma(P,S)$ does not contain a non-zero integer multiple of $\frac{1}{\log m}$.

\item  The measure $g\mu$ is pointwise generic under $T_n$ for an ergodic and continuous measure  $\rho$, and  $f([0,1]), g([0,1])\subseteq [0,1]$.
\end{enumerate}
Then 
\begin{equation} \label{Eq I1 pert}
\frac{1}{N} \sum_{i=0} ^{N-1} \delta_{(T_m ^i f(x) , T_n ^i g(x))} \rightarrow \lambda\times \rho, \quad \text{ for } \mu \text{ almost every } x.
\end{equation}
\end{theorem}

For this to work, we need the following version of Claim \ref{Claim stable intersections}. Let $f,g \in \text{Aff}(\mathbb{R})$. For every $k\in \mathbb{N}$, define
\begin{equation*}
A_k = \lbrace x\in \mathbb{R}: \quad f^{-1}\mathcal{D}_{m^k} (f(x)) \not \subseteq g^{-1}\mathcal{D}_{n^k} (g(x)) \rbrace
\end{equation*}
\begin{Claim} \label{Claim stable intersections pert}
Suppose that $\mu\in \mathcal{P}([0,1])$ is a measure such that $g\mu$  is pointwise generic under $T_n$ for a continuous measure $\rho$. Then for $\mu$ almost every $x$, if $x\in \limsup A_k$ and $\lbrace n_k \rbrace$ represents the times when $x\in A_{n_k}$, then the density of $\lbrace n_k \rbrace$ is zero.
\end{Claim} 
The proof is analogues to that of Claim \ref{Claim stable intersections}.

\textbf{Proof of Theorem \ref{Conjecture pert}}
The proof follows essentially the same steps as the proof of Theorem \ref{Conjecture}. Let $\nu$ be some accumulation point of the orbit under $T_m \times T_n$ of $\delta_{(f(x),g(x))}$, where $x$ is drawn according to $\mu$. 
\begin{itemize}
\item By   (\cite{hochmanshmerkin2015}, Theorem 1.1) we have $P_1 \nu = \lambda$. By our assumption on $g\mu$, $P_2 \nu = \rho$.

\item A complete analogue of Claim \ref{Theorem 5.1} holds in this case as well. First, we disintegrate $P$ according to $\rho$, in a similar manner to the first part of the proof of Claim \ref{Theorem 5.1}. Here, we make use of the fact that $f\mu$ generates and $\log m$ generates $P$, which follows by   (\cite{hochmanshmerkin2015}, Lemma 4.16). 

Secondly, we embed $\mu$ on a line in  $\mathbb{T}^2$ by pushing it forward via the map $x\mapsto (f(x), g(x))$ (recall that we are assuming that both $f$ and $g$ map $[0,1]$ to $[0,1]$). Calling this measure $\tilde{\mu}$, and using the same notation as in Claim \ref{Theorem 5.1}, we have
\begin{equation*}
\nu = \lim_{k\rightarrow \infty} \frac{1}{N_k} \sum_{i=0} ^{N_k-1} \tilde{\mu}_{\mathcal{A}^i \left( f(x),g(x) \right) }
\end{equation*}
by an application of Theorem \ref{Theorem 2.1}. Also, by applying Claim \ref{Claim stable intersections pert}, we see that there is a set $S\subseteq \mathbb{N}$ of density $1$ (possibly depending on the $x$ we chose according to $\mu$), such that for every $i\in S$ the measure $\tilde{\mu}|_{\mathcal{A}^i \left( f(x),g(x) \right) }$ is an affine image of the measure $\mu|_{f^{-1}\mathcal{D}_{m^i} (f(x))}$. Thus, we obtain an analogue of \eqref{Eq analogue}. From here, we complete the proof as in the proof of Claim \ref{Theorem 5.1}.

\item We finish the proof of the Theorem by showing that $\nu$ meets the conditions of Claim \ref{Claim rigidty}. The proof is essentially the same as in the case of Theorem \ref{Conjecture}.
\end{itemize}

\textbf{Proof of Theorem \ref{Main Theorem pert}}
Since we have Theorem \ref{Conjecture pert} at our disposal, the proof is now essentially the same as the proof of Theorem \ref{Main Theorem}. We remark that an analogue of Lemma \ref{Lemma 8.3} remains true in this case as well, which may be deduced from the results of (\cite{hochmanshmerkin2015}, Section 8.4).

\bibliography{bib}{}
\bibliographystyle{plain}

\end{document}